\documentclass{amsart}
\usepackage{}
\usepackage{color}
\usepackage{amssymb}

\newcommand{\s}{\sigma}
\newcommand{\la}{\lambda}

\newtheorem{theorem}{Theorem}[section]
\newtheorem{lemma}[theorem]{Lemma}

 \theoremstyle{definition}

\theoremstyle{remark}

\numberwithin{equation}{section}

\begin{document}

\title[Liouville-type theorem]
{A Liouville-type theorem for $2$-Monge-Amp\`ere equation in 
dimension three}

\author{Weisong Dong}
\address{School of Mathematics, Tianjin University,
	Tianjin, 300354, China}
\email{dr.dong@tju.edu.cn}

\begin{abstract}
	
We prove that every entire solution with quadratic growth, lying in a suitable cone, to the 2-Monge-Amp\`ere equation on $\mathbb{R}^3$ is a quadratic polynomial. The proof proceeds by first establishing a concavity inequality, and then deriving a Pogorelov-type interior $C^2$ estimate.



\end{abstract}

\maketitle

\section{Introduction}

The $p$-Monge-Amp\`ere operator is an important operator that arises in many geometric problems. 
In its simplest form, for a smooth function $u:\mathbb{R}^n \rightarrow \mathbb{R}$, it is defined as follows
\[
m_{p}(\lambda (D^2 u)) := \prod_{1\leq i_1 < \cdots < i_{p} \leq n} 
(\la_{i_1} + \cdots + \la_{i_p}), \; \forall\; 1\leq p\leq n,
\]
where $\lambda(D^2 u) = (\lambda_1, \ldots, \lambda_n)$ are the eigenvalues of the Hessian $D^2u$. When $p=1$, the operator coincides with the Monge-Amp\`ere operator, which arises in the Weyl problem, the Minkowski problem, and in optimal transportation. When $p=n$, it reduces to the standard Laplacian.
The case of $p=n-1$ is particularly important, which arises in the context of the Gauduchon conjecture and the form-type Calabi-Yau equation; see \cite{Gau, GN, P, STW, T-W3} and \cite{FWW1, FWW2}, respectively.
The operator is also closely connected to the concept of 
$p$-plurisubharmonic functions, as introduced by Harvey-Lawson in \cite{HL}.

For solutions to the Laplacian equation $\Delta u = 1$ on $\mathbb{R}^n$, Liouville's theorem tells us that if $u - |x|^2/2n$ is bounded from below, then $u -|x|^2/2n \equiv C$, that is, $u$ is a quadratic polynomial.
Let $u$ be a smooth convex solution to the Monge-Amp\`ere equation $\det(D^2 u) = 1$ 
on $\mathbb{R}^n$. Then, by a classical result of J\"orgens \cite{Jo}, Calabi \cite{Ca} and Pogorelov \cite{Po}, $u$ must be a quadratic polynomial; see also \cite{CYau, TruW}.
This result still holds if $u$ is a viscosity solution; see Caffarelli-Li \cite{CL}.
In this paper, we establish a Liouville-type theorem for smooth entire solution $u$ on $\mathbb{R}^3$ to the following $2$-Monge-Amp\`ere equation
\begin{equation}
	\label{2-MA}
	m_2 (\lambda(D^2 u)) = (\la_1 + \la_2) (\la_1 + \la_3) (\la_2 + \la_3) = 1, 
\end{equation} 
assuming that the solution satisfies a cone condition.
Let $\lambda_{\max}$ and $\lambda_{\min}$ denote the maximum and minimum eigenvalues, respectively, of the Hessian matrix $D^2 u$.
The cone condition is defined as follows
\[
P_2^{1/2} := \Big\{ \lambda  \in \mathbb{R}^3 : \lambda_{i_1} + \la_{i_2} > 0, \; \forall\; i_1 \neq i_2,\; \mbox{and}\; \la_{min} \geq -\frac{1}{2} \la_{max} \Big\}.
\]
We say that a smooth function $u$ is in the $P_2^{1/2}$ cone if $\lambda(D^2 u)\in P_2^{1/2}$.
A function $u:\mathbb{R}^3 \rightarrow \mathbb{R}$ is said to satisfy \emph{quadratic growth condition} if there exist positive constants $C_1$, $C_2$, and $R_0$ such that
\begin{equation}
\label{QG}
u(x) \geq C_1|x|^2 - C_2, \;\forall\; |x| \geq R_0.
\end{equation}

Our main theorem is stated as follows:
\begin{theorem}\label{thm}
	Suppose $u:\mathbb{R}^3 \rightarrow \mathbb{R}$ is a smooth solution in the cone $P_2^{1/2}$ to the equation \eqref{2-MA} satisfying the quadratic growth condition \eqref{QG}. Then, $u$ is a quadratic polynomial.
\end{theorem}

The crucial step here is to prove a concavity inequality for the equation (see Lemma \ref{key}).
It would be interesting to prove the above theorem for solutions lying in the more natural cone defined by
\[
P_2 := \{ \lambda  \in \mathbb{R}^3 : \lambda_{i_1} + \la_{i_2} > 0, \; \forall\; i_1 \neq i_2 \},
\]
since the equation is elliptic for solutions in this cone.
However, the method presented in this paper does not apply to this case.
It would also be interesting to prove the theorem for entire solutions of the $p$-Monge-Amp\`ere equations on $\mathbb{R}^n$ belonging to the cone $P_p$.
Our result generalizes a theorem of Chu and Dinew \cite{CD}, in which the authors proved the statement for semi-convex solutions; that is, there exists a uniform constant $K > 0$ such that $\lambda_{min} \geq - K$.
The main difference in our approach is the establishment of a key concavity inequality for the operator $m_2$ in $\mathbb{R}^3$. Once this is achieved, we derive a Pogorelov-type interior $C^2$ estimate for the solution, following the approach of Li-Ren-Wang \cite{LRW} for the $k$-Hessian equation. The theorem is then proved by adapting the arguments developed by Bao-Chen-Guan-Ji in \cite{BCGJ}. 

A Pogorelov-type interior $C^2$ estimate for the $p$-Monge-Amp\`ere equation in $\mathbb{R}^n$ was established by Dinew \cite{D}, with the bound depending on $||Du||_{C^0}$. Henceforth, he also requires that the solution has quadratic growth from above in order to prove the Liouville-type theorem.
Such an estimate was independently obtained by Chu and Jiao \cite{CJ} for the $n-1$-Monge-Amp\`ere equation, where the right-hand side depends on $Du$.
Dong \cite{D} then generalized the result in \cite{CJ} to the case
$p \geq n/2$. 
In \cite{CJ} and \cite{D}, the authors also proved the existence of $p$-convex hypersurfaces with prescribed curvature in Euclidean space.
The prescribed mean curvature equation for $p$-convex hypersurfaces was explored by Han, Ma, and Wu \cite{HMW}.
In hyperbolic space, a recent paper by Chen, Sui, and Sun \cite{CSS} investigates the Plateau problem for $2$-convex hypersurfaces.

It is also worth mentioning that there are other types of fully nonlinear elliptic equation related to this topic, such as the $k$-Hessian equations, 
\begin{equation}
	\label{sigma-k}
	\s_k (\la (D^2 u) ) = \sum_{1 \leq i_1 < \cdots < i_k \leq n} 
	\la_{i_1} \cdots \la_{i_k} = 1,
\end{equation}
which have been extensively studied in recent years.
The equation is naturally considered in the following G{\aa}rding cone corresponding to the operator, $\Gamma_{k} :=\left\{\lambda \in \mathbb{R}^{n} : \s_{j}(\lambda)>0,\; j=1, \ldots, k\right\}$.
Bao-Chen-Guan-Ji \cite{BCGJ} proved the Liouville-type theorem for convex solutions to the equation \eqref{sigma-k}, assuming a quadratic growth condition from below. Li, Ren, and Wang \cite{LRW} extended the result to solutions lying in the $\Gamma_{k+1}$ cone.
Recently, Zhang \cite{Zhang} further refined the result to the case of semi-convex solutions.
Liu and Ren in \cite{LR} proved a Liouville-type theorem for sum-type Hessian equations. 
A necessary and sufficient condition for locally strictly convex solutions of \eqref{sigma-k} to be quadratic functions was obtained by Du \cite{Du}.
Some counterexamples (see \cite{Li, Warren}) show that the quadratic growth condition is, in general, necessary.
It should be noted that when $k=n-1$, the Liouville-type theorem for \eqref{sigma-k} was proved by Tu \cite{Tu} under the assumption $\lambda(D^2u) \in \Gamma_{n-1}$ and a quadratic growth condition. 
For $k=2$ and $n=3$, it was solved by Chen and Xiang \cite{CX}. 
In a breakthrough paper, Shankar and Yuan \cite{SY} proved a Liouville-type theorem for semiconvex solutions of the equation $\sigma_2(\lambda) = 1$ without assuming the quadratic growth condition,
thereby settling the question posed in \cite{CY}.
The authors further developed their results in \cite{SY2}.
We refer the reader to Fan \cite{Fan} for a generalization to the case of a variable right-hand side function.
A central step in the work \cite{LRW, SY, SY2, Zhang, Fan} is likewise the proof of a concavity inequality (Jacobi inequality). 
We refer the reader to \cite{GQ,LT} for additional results, and to Lu \cite{Lu} for related work on Hessian quotient equations.

The rest of the paper is organized as follows.
In Section 2, we introduce some properties of the $2$-Monge-Amp\`ere operator in $\mathbb{R}^3$. We also prove the key inequality for the $2$-Monge-Amp\`ere equation in $\mathbb{R}^3$.
In Section 3, we prove Theorem \ref{thm}.

\textbf{Acknowledgements}:
This work was carried out while the author was a visitor at the Chern Institute of Mathematics, Nankai University, in 2024, and at IMAG and the Department of Geometry and Topology, University of Granada, in 2025. The author thanks Prof. Weiping Zhang and Prof. José Gálvez for their support, and gratefully acknowledges the warm hospitality of these institutes. His visit to Granada was supported by a scholarship from the China Scholarship Council (No. 202406250051).

\section{Preliminary Lemmas}

This section is devoted to preliminary calculations that will be used in the proof of the main theorem.
Let $u$ be a smooth function on $\mathbb{R}^3$.
For simplicity, we denote 
\[
F (D^2 u) = f (\lambda (D^2 u)) =\left ( \lambda _{1} +\lambda _{2}  \right ) \left ( \lambda _{1} +\lambda _{3}  \right )\left ( \lambda _{2} +\lambda _{3}  \right ).
\]
Let us introduce the following notation:
\[
F^{i j} = \frac{\partial F}{\partial u_{i j}}\;
\mbox{and}\;
F^{i j,r s} = \frac{\partial^2 F}{\partial u_{i j}\partial u_{r s}}.
\]
When $u_{ij}$ is diagonal, we have the following formulas:
\[
F^{i j} =  f_i \delta_{ij},\; \mbox{where}\; f_i = \frac{\partial f}{\partial \lambda_i},
\]
and
\[
F^{i j, r s} \eta_{i j } \eta_{r s } = \sum f_{ij} \eta_{i i } \eta_{j  j } + \sum_{p\neq q}\frac{f_p - f_q}{\lambda_p - \lambda_q} | \eta_{p q}|^2
\]
for any symmetric matrix $\{\eta_{ij}\}$. See \cite{A, Ball, G, Spruck}.

The following formulas will be useful in our computations:
\begin{lemma}
	\label{Df}
	For any $1\leq i, j, k \leq 3$ and $i \neq j \neq k \neq i$, we have
	\[
	f_i = \sigma_1^2 - \la_i^2,
	\]
	\[
	f_{ii} = 2 (\sigma_1 - \la_i ),
	\]
	\[
	 f_{ij} = \frac{f_{i}f_{j}}{f} - \frac{f}{\left(\lambda_{i} + \lambda_{j}\right)^2},
	\]
	and
	\[
	 f_{ii} = \frac{f_{i}^{2}}{f} - \frac{f}{\left(\lambda_{i} + \lambda_{j}\right)^2} - \frac{f}{\left(\lambda_{i} + \lambda_{k}\right)^2}.
	\]
\end{lemma}

\begin{proof}
Without loss of generality, we shall prove the lemma for the case $i = 1$ and $j = 2$.
Recall that 
\[
\s_1 = \la_1 + \la_2 + \la_3.
\]
The first order derivative of $f$ with respect to $\la_1$ is given as below
\[
	f_{1} 
	=  \frac{f}{\la_1 + \la_2} + \frac{f}{\la_1 + \la_3}
	= \sigma_{1}^2 - \lambda_{1}^2.
\]
The second order derivatives are given by
\[
f_{11} = \frac{2f}{( \lambda_{1} + \lambda_{2} ) ( \lambda_{1} + \lambda_{3} )} = 2 (\sigma_1 -\la_1)
\]
and
\[
f_{12} = 2 \s_1.
\]
A direct calculation shows that
\[
f_{1}^{2} = \frac{f^{2}}{\left ( \lambda_{1} + \lambda_{2}\right )^2} + \frac{f^{2}}{\left ( \lambda_{1} + \lambda_{3}\right )^2} + f_{11}\cdot f.
\]
Note that
\[
f_{2} = \frac{f}{\lambda_{2} + \lambda_{3} } + \frac{f}{\lambda_{1} + \lambda_{2}}.
\]
We obtain that
\[
f_{1} f_{2} = \frac{f^2}{(\la_1 + \la_2) (\la_2 + \la_3)} +
\frac{f^2}{(\lambda_1 + \lambda_2)^2} 
+ \frac{f^2}{(\la_1 + \la_3) (\la_2 + \la_3)} + \frac{f^2}{(\la_1 + \la_2) (\la_1 + \la_3)}.
\]
Hence, we conclude that
\[
\frac{f_{1} f_{2}}{f} 
 =  2\sigma_{1} + \frac{f}{\left(\lambda_{1} + \lambda_{2} \right)^2}. 
\]
The lemma is proved.

\end{proof}

Now we introduce the following notation:
\[
b_{ii} :=\sum_{k\ne i}\frac{f}{\left(\lambda_{i}+\lambda_{k}\right)^2},
\]
and
\[
b_{ij} :=\frac{f}{\left(\lambda_{i}+\lambda_{j}\right)^2} = b_{ji}, \;
\mbox{for}\; i \ne j.
\]
Then, by Lemma \ref{Df}, at a point where $D^2 u$ is diagonal, we have
\begin{equation}
	\label{fii}
\frac{F^{ii}F^{ii}}{F} = F^{ii,ii}+b_{ii}, 
\end{equation}
and
\begin{equation}
	\label{fij}
\frac{F^{ii}F^{jj}}{F} = F^{ii,jj}+b_{ij}, \; \mbox{for} \; i \ne j.
\end{equation}

\begin{lemma}
	\label{key-2}
    Suppose that $\lambda (D^2 u) = (\la_1, \la_2, \la_3) \in P_2^{1/2}$ 
    satisfies equation \eqref{2-MA}. Without loss of generality, assume $\la_1 = \max_{1\leq i \leq 3}{\la_i}$. 
	Then
	\[
	0 < \lambda_2 + \lambda_3 \leq \frac{4}{\lambda_1^2}.
	\]
\end{lemma}

\begin{proof}
	We have
	\[
	1 = (\la_1 + \la_2)(\la_1 + \la_3)(\la_2 + \la_3) \geq \frac{\la_1^2}{4} (\la_2 + \la_3).
	\]
	The lemma is proved.
\end{proof}

We now establish our key concavity inequality.
\begin{lemma}
	\label{key}
    Suppose that $\lambda (D^2 u) = (\la_1, \la_2, \la_3) \in P_2^{1/2}$ 
    satisfies equation \eqref{2-MA}. Without loss of generality, assume $\la_1 = \max_{1\leq i \leq 3}{\la_i}$. Then, at a point where $D^2u$ is diagonal, for sufficiently large $\la_1$, we have
	\[
	- \sum F^{pp,qq} \xi_p \xi_q + 2\sum_{i > 1} \frac{F^{ii}\xi_i^2}{\lambda_1 -\lambda_i} - (1+\epsilon) \frac{F^{11} \xi_1^2}{\lambda_1} \geq 0,
	\]
	where $\epsilon = 1/9$ and $\xi = \left(u_{111},u_{221},u_{331}\right)$.
\end{lemma}

\begin{proof}
Let $\delta = 1 + \epsilon$.
A direct calculation shows that
\begin{equation}
\begin{aligned}
Q
:=&\; -\sum F^{pp,qq}\xi_{p}\xi_{q} + 2\sum_{i> 1}\frac{F^{ii}\xi_{i}^{2}}{\lambda_{1}-\lambda_{i}}-\delta \frac{F^{11}\xi_{1}^{2}}{\lambda _{1} } \\
=&\; - \sum_{p\ne q}F^{pp,qq}\xi_{p}\xi_{q} - \sum_{p}F^{pp,pp}\xi_{p}^{2} + 2\sum_{i> 1}\frac{F^{ii}\xi_{i}^{2}}{\lambda_{1}-\lambda_{i}}-\delta \frac{F^{11}\xi_{1}^{2}}{\lambda_{1} }\\       
=&\; - \sum_{p\ne 1}F^{pp,11}\xi_{p}\xi_{1} - \sum_{q\ne 1}F^{11,qq}\xi_{1}\xi_{q} - F^{11,11}\xi_{1}^{2} - \delta \frac{F^{11}}{\lambda_{1}}\xi_{1}^{2}\\
&\; - \sum_{p\ne q\ p,q\ge 2}F^{pp,qq} \xi_{p} \xi_{q} - \sum_{p\ge 2}F^{pp,pp}\xi_{p}^{2} + 2\sum_{i> 1}\frac{F^{ii}\xi_{i}^{2} }{\lambda_{1}-\lambda_{i} }.
\end{aligned}
\end{equation}
Recall that $\xi = \left(u_{111},u_{221},u_{331}\right)$ and
$F^{11} u_{111}+F^{22}u_{221}+F^{33}u_{331}=0$ at a point where $D^2u$ is diagonal. We see that
\[
\xi_{1}
=-\frac{F^{22}\xi_{2}+F^{33}\xi_{3}}{F^{11}}.
\]
Substituting the above into $Q$, we obtain that
\[
\begin{aligned}  	           
Q =&\;  \sum_{p\ne 1}F^{pp,11}\xi _{p} \sum_{q\ge 2} \frac{F^{qq}\xi _{q}}{F^{11}} + \sum_{q\ne 1}F^{11,qq}\xi_{q} \sum_{p\ge 2} \frac{F^{pp}\xi _{p}}{F^{11}}\\
&\; - F^{11,11} \left( \sum_{p\ge 2} \frac{ F^{pp} \xi_{p}}{F^{11} }  
\right)^2 - \delta \frac{1}{\lambda _{1}F^{11}} \left( \sum_{p\ge 2}F^{pp}\xi _{p} \right)^2 \\
&\;- \sum_{p\ne q; p,q \geq 2} F^{pp,qq}\xi _{p}\xi _{q} - \sum_{p\ge 2} F^{pp,pp}\xi _{p}^{2} + 2\sum_{i> 1}\frac{F^{ii}\xi _{i}^{2}}{\lambda _{1} -\lambda_{i}}.
\end{aligned}
\]
Let us introduce the following notation
\[
Q =: \sum_{p\ne q; p, q \geq 2}a_{pq}\xi _{p}\xi _{q} + \sum_{p \geq 2} a_{pp}\xi _{p}^{2},
\]
where, for $p\ne q$,
\[\begin{aligned}
a_{pq}
=&\;  \frac{F^{pp,11} F^{qq}}{ F^{11} } + \frac{F^{11,qq}F^{pp}}{F^{11}} - F^{11,11}\frac{F^{pp} F^{qq} }{\left ( F^{11}  \right )^2 }\\
&\; - \delta \frac{1}{\lambda _{1}F^{11}}F^{pp} F^{qq} - F^{pp,qq}
\end{aligned}
\]
and
\[\begin{aligned}
a_{pp}
= &\; \frac{F^{pp,11} F^{pp}}{ F^{11} } + \frac{F^{11,pp}F^{pp}}{F^{11}} - F^{11,11} \left(\frac{ F^{pp} }{ F^{11} } \right)^2
\\
&\; - \delta \frac{1}{\lambda _{1}F^{11}}\left ( F^{pp} \right )^2 - F^{pp,pp} + 2\frac{F^{pp}}{\lambda _{1}-\lambda _{p}}.
\end{aligned}\]

We split the proof into two steps.
First, we show that \textbf{Step 1:} $a_{22} \geq a_{23}$, and then we prove that \textbf{Step 2:} $a_{22} \geq - a_{23}$.
By a similar argument, it follows that 
\[
a_{33} \geq a_{32}\; \mbox{and} \; a_{33} \geq - a_{32}.
\]
Therefore, the quadratic form $Q$ is non negative, since 
\[
a_{22}\geq 0, \; a_{33} \geq 0\; \mbox{and}\; \det (a_{pq}) \geq 0,
\]
which completes the proof of the lemma.

\textbf{Step 1:} We show $a_{22} \geq a_{23}$. We have
\[
\begin{aligned}
a_{22}-a_{23}
=&\; \frac{F^{22,11}}{F^{11}}\left ( F^{22} - F^{33}\right ) + \frac{F^{22}}{F^{11} }\left ( F^{11,22} -F^{11,33} \right ) \\
&\; - F^{11,11}\frac{F^{22}\left ( F^{22}-F^{33}\right )}{\left ( F^{11}\right )^2 } + \frac{ \delta  }{\lambda _{1} F^{11}}F^{22}\left ( F^{33}-F^{22}\right )\\
&\; - F^{22,22} + F^{22,33} + 2\frac{F^{22}}{\lambda _{1} -\lambda _{2} }.
\end{aligned}
\]
By $f_{13} = 2\sigma_{1} = f_{12}$, we obtain that
\begin{equation}
	\label{2-3-1}
	\begin{aligned}
	a_{22}-a_{23}
	= &\; \left(F^{22}-F^{33}\right) \left\{ \frac{F^{22,11}}{F^{11}} - \frac{F^{11,11}F^{22}}{\left(F^{11}\right)^2} - \frac{\delta}{\lambda_1 F^{11}}F^{22} \right\}\\
	&\; - F^{22,22} + F^{22,33} + 2\frac{F^{22}}{\lambda_{1} -\lambda_{2}}.
\end{aligned}
\end{equation}
Recall that
\[
\frac{F^{22,11}}{F^{11}}
= \frac{F^{22}}{ F} - \frac{b_{21}}{F^{11}}
\]
and
\[
\frac{F^{11,11}}{(F^{11})^2} = \frac{1}{F} - \frac{b_{11}}{(F^{11})^2}.
\]
It is also straightforward to compute that
\[
F^{22,33} - F^{22,22} = \frac{F^{22} (F^{33} - F^{22})}{F} + b_{22} - b_{23}.
\]
By the above three equations, we obtain from \eqref{2-3-1} that
\begin{equation}
	\label{2-3-2'}
	\begin{aligned}
		a_{22}-a_{23}
		= &\; F^{22}\left(F^{22}-F^{33}\right) \Big\{ \frac{1}{F} - \frac{2}{F} + \frac{b_{11}}{(F^{11})^2} - \frac{\delta}{\lambda_1 F^{11}}  \Big\}\\
		&\; - \frac{b_{21}}{F^{11}} \left(F^{22}-F^{33}\right) 
		 + b_{22} - b_{23} + 2\frac{F^{22}}{\lambda_{1}-\lambda _{2}}.
	\end{aligned}
\end{equation}

We now group the terms on the right-hand side into four categories,
\[
I := b_{22} - b_{23},
\]
\[
II := - \frac{b_{21}}{F^{11}} \left(F^{22}-F^{33}\right), 
\]
\[
III := 2\frac{F^{22}}{\lambda_{1}-\lambda _{2}},
\]
and
\[
IV := F^{22}\left(F^{22}-F^{33}\right) \left\{ \frac{b_{11}}{(F^{11})^2} - \frac{\delta}{\lambda_1 F^{11}} - \frac{1}{F} \right\}.
\]
A direct computation further shows that
\[
I = \frac{f}{(\la_1 + \la_2)^2} = \frac{(\la_1 + \la_3)(\la_2 + \la_3)}{\la_1 + \la_2},
\]
and, by Lemma \ref{Df},
\[\begin{aligned}
II = \frac{(\la_1 + \la_3)(\la_2 + \la_3)( \la_2^2 - \la_3^2)}{(\la_1 + \la_2) (\s_1^2 - \la_1^2)} 
=  \frac{(\la_1 + \la_3)(\la_2 + \la_3)(\la_2 - \la_3)}{(\la_1 + \la_2) (\s_1 + \la_1)}.
\end{aligned}\]
By adding $I$ and $II$, we obtain
\[
I + II 
= \frac{(\la_1 + \la_3)(\la_2 + \la_3)(\s_1 + \la_1 + \la_2 - \la_3)}{(\la_1 + \la_2) (\s_1 + \la_1)}
= 2 \frac{(\la_1 + \la_3)(\la_2 + \la_3)}{ \s_1 + \la_1}.
\]

We now proceed with some calculations. By the definition of $b_{ij}$ and $f$, we obtain that
\[
\begin{aligned}
	fb_{11}
	=&\;\left ( \lambda_{2}  + \lambda_{3}  \right )^2 [ \left ( \lambda_{1}  + \lambda_{3}  \right )^2 + \left ( \lambda_{1}  + \lambda_{2}  \right )^2 ]\\
	=&\;\left ( \lambda_{2}  + \lambda_{3}  \right )^2 [ \left ( \sigma_{1} + \lambda_{1}\right )^2-2 \left ( \lambda_{1} + \lambda_{3}\right) \left(\lambda_{1} + \lambda_{2}\right) ]\\
	=&\;\left (\sigma_{1} - \lambda_{1}\right )^2\left ( \sigma_{1}  +\lambda_{1}\right )^2-2f\left ( \lambda_{2} + \lambda_{3}\right )\\
	=&\;\left (\sigma_{1} - \lambda_{1}\right )^2\left ( \sigma_{1}  +\lambda_{1}\right )^2 - 2f\left ( \sigma_{1} - \lambda_{1}\right ).
\end{aligned}
\]
Hence, we have
\[\begin{aligned}
\frac{b_{11}}{\left ( F^{11}\right)^2} 
= &\; \frac{1}{f} - \frac{2}{(\s_1 - \la_1) (\s_1 + \la_1)^2}. 
\end{aligned}\]
Note that
\[
\frac{\delta}{\lambda _{1}F^{11}}=\frac{\delta}{\lambda _{1}\left ( \sigma _{1}^{2}-\lambda _{1}^{2}\right)} 
= \frac{\delta }{\lambda _{1}\left ( \sigma _{1}  +\lambda _{1}\right )\left ( \sigma _{1}  -\lambda _{1}\right ) }.
\] 
Therefore, it follows that
\begin{equation}
	\label{delta}
\frac{b_{11}}{(F^{11})^2} - \frac{\delta}{\lambda_1 F^{11}} - \frac{1}{F} 
= - \frac{ 2\lambda_{1} + \delta \left(\sigma_{1} + \lambda_{1}\right)  }{\lambda_{1}\left(\sigma_{1} - \lambda_{1}\right)\left(\sigma_{1} + \lambda_{1}\right)^2}.
\end{equation}
Substituting the above into the term $IV$, we obtain
\[
IV = (\s_1^2 -\la_2^2) (\la_2^2 - \la_3^2) \frac{ 2\lambda_{1} + \delta \left(\sigma_{1} + \lambda_{1}\right)  }{\lambda_{1}\left(\sigma_{1} - \lambda_{1}\right)\left(\sigma_{1} + \lambda_{1}\right)^2}.
\]
Thus, we arrive at
\[
I + II + IV = \frac{1}{\s_1 + \la_1} \left\{2 (\la_1 + \la_3)(\la_2 + \la_3) + (\s_1^2 -\la_2^2) (\la_2 - \la_3) \frac{ 2\lambda_{1} + \delta \left(\sigma_{1} + \lambda_{1}\right)  }{\lambda_{1} \left(\sigma_{1} + \lambda_{1}\right)}
\right\}.
\]
Observe that the first term within the curly brackets is positive.
It follows that
\[\begin{aligned}
&\; I + II + III + IV \\
\geq &\; \frac{\s_1^2 -\la_2^2}{\s_1 + \la_1} \left\{ 2\frac{\s_1 + \la_1}{\lambda_{1} - \lambda _{2}} + (\la_2 - \la_3) \frac{ 2\lambda_{1} + \delta \left(\sigma_{1} + \lambda_{1}\right)  }{\lambda_{1} \left(\sigma_{1} + \lambda_{1}\right)}
\right\}.
\end{aligned}\]
The expression inside the curly brackets can be computed as follows:
\[
\begin{aligned}
	&\; 2 \frac{\sigma_{1}+\lambda_{1} }{\lambda_{1}-\lambda_{2}} + (\lambda_{2}-\lambda_{3}) \frac{ 2\lambda_{1} + \delta (\lambda_{1}+\sigma_{1}) }{\lambda_{1} (\sigma_{1} + \lambda_{1} )}\\
	= &\;\frac{\sigma_{1}+\lambda_{1}}{\lambda_{1}-\lambda_{2}} + 2 \frac{ \lambda_{2}-\lambda_{3} }{\sigma_{1}+\lambda_{1}}+\frac{\sigma_{1}+\lambda_{1}}{\lambda_{1}-\lambda_{2}} + \delta \frac{ \lambda_{2} -\lambda_{3} }{\lambda_{1}}\\
	= &\;\frac{V}{\left(\lambda_{1} - \lambda_{2}\right)\left(\lambda_{1} + \sigma_{1}\right)} + \frac{VI}{\left(\lambda_{1} - \lambda_{2}\right)\lambda_{1}},
\end{aligned}
\]
where we define
\[
V := \left(\sigma_{1}+\lambda_{1}\right)^2+2\left(\lambda_{2}-\lambda_{3}\right)\left(\lambda_{1}-\lambda_{2}\right)
\]
and
\[
VI := \left(\sigma_{1}+\lambda_{1}\right)\lambda_{1}+\delta\left(\lambda_{2}-\lambda_{3}\right)\left(\lambda_{1}-\lambda_{2}\right).
\]

We now divide the analysis into three cases.

\textbf{Case 1:} $\la_2, \la_3 \geq 0$.
Computing directly, we find that
\[
\begin{aligned}
	V = &\;4\lambda_{1}^{2}+6\lambda_{1}\lambda_{2}+2\lambda_{1}\lambda_{3}-\lambda_{2}^{2}+4\lambda_{2}\lambda_{3}+\lambda_{3}^{2}\\
	= &\; 4\lambda_{1}^{2}-\lambda_{2}^{2}+\lambda_{3}^{2}+2\lambda_{1}\left(\lambda_{2}+\lambda_{3}\right)+4\lambda_{2}\left(\lambda_{1}+\lambda_{3}\right)> 0\\
\end{aligned}	
\]
since $\lambda_1 \geq \lambda_2$. For $VI$, note that
\[
\begin{aligned}
	VI
	=&\;2\lambda_{1}^{2}+\left(1+\delta\right)\lambda_{1}\lambda_{2}+\left(1-\delta\right)\lambda_{1}\lambda_{3}-\delta\lambda_{2}^{2}+\delta\lambda_{2}\lambda_{3}\\ 
	\geq &\; 2\lambda_{1}^{2} -\left( 1 + \epsilon\right) \lambda_{2}^{2} - \epsilon \lambda_{1} \lambda_{3} > 0 
\end{aligned}
\]
for $\epsilon < 1/2$.

\textbf{Case 2:} $\lambda_{3}< 0$. Observe that
\[
V = \left(\sigma_{1}+\lambda_{1}\right)^2+2\left(\lambda_{2}-\lambda_{3}\right)\left(\lambda_{1}-\lambda_{2}\right) > 0
\]
and
\[
VI= \left(\sigma_{1}+\lambda_{1}\right)\lambda_{1}+\delta\left(\lambda_{2}-\lambda_{3}\right)\left(\lambda_{1}-\lambda_{2}\right) > 0.
\]

\textbf{Case 3:} $\lambda_2 < 0$.  
Note that $\lambda_2 \geq - \frac{1}{2} \lambda_1$ and $\la_3 > - \la_2 > 0$.
We obtain that
\[\begin{aligned}
	V \geq 4\lambda_{1}\lambda_{2}+4\lambda_{2}\lambda_{3}+4\lambda_{1}^{2}
	\geq  2 \lambda_1^2 + 4 \lambda_2 \lambda_3 > 0
\end{aligned}
\]
since $\lambda_1 > \lambda_3$.
On the other hand, for $VI$, we have 
\[\begin{aligned}
	VI \geq &\; 2\lambda_{1}^{2} - \frac{2+\epsilon}{2}\lambda_{1}^{2} - \epsilon\lambda_{1}^{2} - \frac{1+\epsilon}{4}\lambda_{1}^{2} - \frac{1+\epsilon}{2}\lambda_{1}^{2}\\
	=&\; 2\lambda_{1}^{2}-\frac{7+9\epsilon}{4}\lambda_{1}^{2} \geq 0,
\end{aligned}     
\]
if $\epsilon \leq 1/9$.
This completes the proof of Step 1.

\textbf{Step 2:}
We prove $a_{22} \geq - a_{23}$.
It is straightforward to verify that
\begin{equation}
	\label{2+3-1}
\begin{aligned}
	&\; a_{22} + a_{23}\\
	=&\; \frac{F^{22,11}}{F^{11}}\left ( F^{22} + F^{33} \right ) + \frac{F^{22}}{F^{11} }\left ( F^{11,22} + F^{11,33} \right ) \\
	&\; - F^{11,11}\frac{F^{22}\left ( F^{22} + F^{33}\right )}{\left ( F^{11}\right )^2 }
	 - \frac{\delta}{\lambda _{1} F^{11}}F^{22}\left ( F^{22} + F^{33} \right ) \\
	 &\; - F^{22,22} - F^{22,33} + 2\frac{F^{22}}{\lambda _{1} -\lambda _{2} }.
\end{aligned}
\end{equation}
That is
\begin{equation}
	\label{2+3-2}
	\begin{aligned}
		& a_{22} + a_{23}\\
		=&\; 
		F^{22}\left ( F^{22} + F^{33}\right ) \left\{ - \frac{F^{11,11}}{\left ( F^{11}\right )^2 } - \frac{\delta}{\lambda_{1} F^{11}}\right\}  
		- F^{22,22} - F^{22,33}\\
		&\; + \frac{F^{22,11}}{F^{11}} \left(F^{22} + F^{33}\right) + \frac{F^{22}}{F^{11}}\left ( F^{11,22} + F^{11,33} \right ) + 2\frac{F^{22}}{\lambda_{1}  -\lambda_{2}}.  
	\end{aligned}
\end{equation}
By \eqref{fii} and \eqref{fij}, we obtain that
\[
F^{22,22} + F^{22,33} 
= \frac{F^{22}}{F}\left(F^{22} + F^{33}\right)
	- b_{22} - b_{23}.
\]
Similarly, we find that
\[
	 \frac{F^{22}}{F^{11}} (F^{11,22} + F^{11,33})
	=  \frac{F^{22}}{F} \left( F^{22} + F^{33} \right)  - \frac{F^{22}}{F^{11}} (b_{12} + b_{13} ).
\]
Note that
\[
\frac{F^{22,11}}{F^{11}}
= \frac{F^{22}}{ F} - \frac{b_{21}}{F^{11}}.
\]
Based on the above three equations, we obtain 
\begin{equation}
	\label{2+3-3}
\begin{aligned}
	&\; a_{22} + a_{23}\\
= &\; F^{22}\left ( F^{22} + F^{33}\right ) \left\{ - \frac{F^{11,11}}{\left ( F^{11}\right )^2 } - \frac{\delta}{\lambda_{1} F^{11}} - \frac{1}{F}\right\} + b_{22} + b_{33}\\
&\; + 2 \frac{F^{22}}{F}\left(F^{22} + F^{33}\right) - \frac{b_{21}}{F^{11}}\left(F^{33} + F^{22}\right) - \frac{F^{22}}{F^{11}} b_{11} + 2\frac{F^{22}}{\lambda_{1} - \lambda_{2}},
\end{aligned}
\end{equation}
where we used $b_{11} = b_{12} + b_{13}$.
Substituting
$F^{11,11}+b_{11} = (F^{11})^2/F$
into \eqref{2+3-3}, it follows that
\begin{equation}
	\label{2+3-4}
\begin{aligned}
	&\; a_{22} + a_{23}\\
	= &\; F^{22}\left ( F^{22} + F^{33}\right ) \left\{  \frac{b_{11}}{\left ( F^{11}\right )^2 } - \frac{\delta}{\lambda_{1} F^{11}} \right\} + b_{22} + b_{33}\\
	&\; - \frac{b_{21}}{F^{11}}\left(F^{33} + F^{22}\right) 
	- \frac{F^{22}}{F^{11}} (b_{12} + b_{13}) + 2\frac{F^{22}}{\lambda_{1}  -\lambda_{2}}  
\end{aligned}
\end{equation}
since
$b_{11} = b_{12} + b_{13}$.

We now group the terms on the right-hand side into four categories,
\[
I : = b_{22} +  b_{23} - \frac{F^{22}}{F^{11}} (b_{12} + b_{13}),
\]
\[
II :=- \frac{b_{21}}{F^{11}}\left( F^{22} + F^{33}\right),
\]
\[
III : = 2 \frac{F^{22}}{\lambda_{1}  -\lambda_{2}},
\]
and
\[
IV : = F^{22} \left (F^{22} + F^{33} \right ) \left\{ \frac{b_{11}}{\left( F^{11}\right)^2} - \frac{\delta}{\lambda_{1}F^{11}} \right\}.
\]
We see that
\begin{equation}
	\label{2+3-5}
		 a_{22} + a_{23} > I + II + IV,
\end{equation}
since $III > 0$.
A routine computation yields
\[\begin{aligned}
b_{22} + b_{23} 
= \frac{(\lambda_1 + \lambda_3) (\sigma_1 - \lambda_1)}{\lambda_1 + \lambda_2} + 2 \frac{(\lambda_1 + \lambda_2) (\lambda_1 + \lambda_3)}{\lambda_2 + \lambda_3}
\end{aligned}\]
and
\[
\frac{F^{22}}{F^{11}} (b_{12} + b_{13})
= \frac{(\sigma_1^2 -\lambda_2^2) (\lambda_1 + \lambda_3) }{(\sigma_1 + \lambda_1)(\lambda_1 + \lambda_2)} + \frac{(\sigma_1 + \lambda_2) (\lambda_1 + \lambda_2)}{\sigma_1 + \lambda_1}.
\]
Combining the two equations above, we obtain
\[
\begin{aligned}
	I 
	= \frac{(\lambda_2^2 - \lambda_1^2) (\lambda_1 + \lambda_3) }{(\sigma_1 + \lambda_1)(\lambda_1 + \lambda_2)} + 2 \frac{(\lambda_1 + \lambda_2) (\lambda_1 + \lambda_3)}{\lambda_2 + \lambda_3} - \frac{(\sigma_1 + \lambda_2) (\lambda_1 + \lambda_2)}{\sigma_1 + \lambda_1 }.
\end{aligned}
\]
For $II$, it follows by direct computation that
\[\begin{aligned}
II = &\; - \frac{(\lambda_1 + \lambda_3) (\lambda_2 + \lambda_3)}{(\lambda_1 + \lambda_2) (\sigma_1^2 - \lambda_1^2)} \left(\sigma_1^2 - \lambda_2^2 + \sigma_1^2 -\lambda_3^2\right)\\
= &\;  - \frac{\lambda_1 + \lambda_3}{(\lambda_1 + \lambda_2) (\sigma_1 + \lambda_1)} \left(\sigma_1^2 - \lambda_2^2 + \sigma_1^2 -\lambda_3^2\right)\\
= &\; - \frac{(\sigma_1^2 -\lambda_2^2)(\lambda_1 + \lambda_3)}{(\lambda_1 + \lambda_2) (\sigma_1 + \lambda_1)} - \frac{(\sigma_1 + \lambda_3) (\lambda_1 + \lambda_3)}{ \sigma_1 + \lambda_1}.
\end{aligned}\]
Combining $I$ and $II$, we find that
\[
\begin{aligned}
	I + II = &\; \frac{( \lambda_2^2 - \lambda_1^2 ) (\lambda_1 + \lambda_3) }{(\sigma_1 + \lambda_1)(\lambda_1 + \lambda_2)} 
	- \frac{(\sigma_1^2 - \la_2^2) (\lambda_1 + \lambda_3) }{(\sigma_1 + \lambda_1)(\lambda_1 + \lambda_2)}\\
	&\; + 2 \frac{(\lambda_1 + \lambda_2) (\lambda_1 + \lambda_3)}{\lambda_2 + \lambda_3} - \frac{(\sigma_1 + \lambda_2) (\lambda_1 + \lambda_2)}{\sigma_1 + \lambda_1} - \frac{(\sigma_1 + \lambda_3) (\lambda_1 + \lambda_3)}{ \sigma_1 + \lambda_1}.
\end{aligned}
\]
Note that
\[\begin{aligned}
\frac{\lambda_1 + \lambda_3}{\lambda_2 + \lambda_3} - \frac{\sigma_1 + \lambda_2}{\sigma_1 + \lambda_1}
 = \frac{2 \sigma_1 (\lambda_1 -\lambda_2)}{\sigma_1^2 - \lambda_1^2},
\end{aligned}\]
and, similarly,
\[
 \frac{\lambda_1 + \lambda_2}{\lambda_2 + \lambda_3} - \frac{\sigma_1 + \lambda_3}{\sigma_1 + \lambda_1}
 = \frac{2 \sigma_1 (\lambda_1 -\lambda_3)}{\sigma_1^2 - \lambda_1^2}.
\]
Consequently, we find
\[
\begin{aligned}
	I + II
	= &\; \frac{( \lambda_2 - \lambda_1 ) (\lambda_1 + \lambda_3) (\la_2 + \la_3) }{\sigma_1^2 - \lambda_1^2} 
	- \frac{(\sigma_1^2 - \lambda_2^2) (\s_1 - \la_2) }{(\sigma_1 + \lambda_1)(\lambda_1 + \lambda_2)} \\
	&\; + \frac{2 \sigma_1 (\lambda_1 + \lambda_2) (\lambda_1 -\lambda_2)}{\sigma_1^2 - \lambda_1^2} + \frac{2 \sigma_1 (\lambda_1 + \lambda_3) (\lambda_1 -\lambda_3)}{\sigma_1^2 - \lambda_1^2}.
\end{aligned}
\]
By \eqref{delta}, we obtain that
\[\begin{aligned}
IV 
= &\; \frac{(\sigma_1^2 - \lambda_2^2)(\sigma_1 + \lambda_2)}{(\lambda_1 + \lambda_2) (\lambda_2 + \lambda_3)} + \frac{(\sigma_1 + \lambda_2) (\sigma_1 + \lambda_3)}{\lambda_2 + \lambda_3} \\
 & \; - \left[ (\sigma_1^2 - \lambda_2^2)^2 + (\sigma_1^2 - \lambda_2^2) (\sigma_1^2 - \lambda_3^2) \right] \frac{2 \lambda_1 + \delta (\sigma_1 + \la_1) }{\lambda_1 (\sigma_1 -\lambda_1) (\sigma_1 + \lambda_1)^2}.
\end{aligned}\]
That is,
\[
\begin{aligned}
	\uppercase\expandafter{\romannumeral4} 
	= &\; \frac{(\sigma_1^2 - \lambda_2^2)(\sigma_1 + \lambda_2) (\sigma_1 + \lambda_1)}{(\lambda_1 + \lambda_2) (\sigma_1^2 - \lambda_1^2)} + \frac{(\sigma_1 + \lambda_2) (\sigma_1 + \lambda_3) (\s_1 + \la_1)}{\s_1^2 - \la_1^2} \\
	& \; - \left[ (\sigma_1^2 - \lambda_2^2)^2 + (\sigma_1^2 - \lambda_2^2) (\sigma_1^2 - \lambda_3^2) \right] \frac{2 \lambda_1 + \delta ( \sigma_1 + \la_1) }{\lambda_1 (\sigma_1^2 -\lambda_1^2) (\sigma_1 + \lambda_1)}.
\end{aligned}
\]
Combining $I$, $II$, and $IV$, we derive
\[
\begin{aligned}
	&\; I + II + IV\\
	= &\; \frac{1}{\s_1^2 - \la_1^2} \Big\{ (\la_2 -\la_1) (\la_1 + \la_3) (\la_2 + \la_3) + 2 \s_1 (\la_1^2 -\la_2^2) + 2 \s_1 (\la_1^2 - \la_3^2) \\
	&\; - \frac{(\s_1^2 - \la_2^2) (\s_1 - \la_2) (\s_1 - \la_1)}{\la_1 + \la_2} + \frac{(\s_1^2 - \la_2^2) (\s_1 + \la_2) (\s_1 + \la_1)}{\la_1 + \la_2}\\
	&\; - \frac{2 \la_1 + \delta (\s_1 + \la_1 )}{\la_1 (\s_1 + \la_1)}
	\left[(\s_1^2 - \la_2^2)^2 + (\s_1^2 - \la_2^2) (\s_1^2 - \la_3^2) \right] \\
	&\; + (\s_1 + \la_2) (\s_1 + \la_3) (\s_1 + \la_1) \Big\}.
\end{aligned}
\]
For the last term in the above equation, a direct computation yields
\[
\begin{aligned}
&\; (\s_1 + \la_2) (\s_1 + \la_3) (\s_1 + \la_1)\\
= &\; (\la_1 + \la_2 + \la_2 + \la_3) (\s_1 + \la_3) (\s_1 + \la_1)\\
\geq &\; (\la_1 + \la_2) (\la_1 + \la_3) (\s_1 + \la_1) + (\la_2 + \la_3) (\s_1 + \la_3) (\la_1 + \la_3).
\end{aligned}
\]
Note that
\[
(\la_1 + \la_3) (\la_2 + \la_3) (\s_1 + \la_3) + (\la_2 -\la_1) (\la_1 + \la_3) (\la_2 + \la_3) > 0.
\]
It is also easy to see that
\[
\frac{(\s_1^2 - \la_2^2) (\s_1 + \la_2) (\s_1 + \la_1)}{\la_1 + \la_2}
- \frac{(\s_1^2 - \la_2^2) (\s_1 - \la_2) (\s_1 - \la_1)}{\la_1 + \la_2} 
= 2 \s_1 (\s_1^2 - \la_2^2).
\]
Hence, we have
\[
\begin{aligned}
	&\; I + II + IV\\
	\geq &\; \frac{1}{\s_1^2 - \la_1^2} \left\{ (\la_1 + \la_2) (\s_1 - \la_2) (\s_1 + \la_1) \right.\\
	&\; \left. + 2 \s_1 (\s_1 -\la_3) (\la_1 - \la_2) + 2 \s_1 (\s_1 - \la_2) (\la_1 - \la_3) \right.\\
	&\; \left. + (\s_1^2 - \la_2^2) \Big[ 2 \s_1 - \Big(\frac{2}{ \s_1 + \la_1} + \frac{\delta}{\la_1} \Big) (\s_1^2 - \la_2^2 + \s_1^2 - \la_3^2) \Big] \right\}.
\end{aligned}
\]
Note that
\[
\begin{aligned}
	&\; 2 \s_1 - \frac{2}{ \s_1 + \la_1} (\s_1^2 - \la_2^2 + \s_1^2 - \la_3^2)\\
	= &\; \frac{1}{\s_1 + \la_1} [ 2 \s_1 (\s_1 + \la_1) - 2 (2\s_1^2 - \la_2^2 - \la_3^2) ]\\
	= &\; \frac{1}{\s_1 + \la_1} [ 2 \s_1 (\la_1 - \s_1) + 2\la_2^2 + 2 \la_3^2 ].
\end{aligned}
\]
This shows that
\[
\begin{aligned}
	&\; I + II + IV\\
	\geq &\; \frac{1}{\s_1^2 - \la_1^2} \Big\{ (\la_1 + \la_2) (\s_1 - \la_2) (\s_1 + \la_1) +
	2 \s_1 (\s_1 -\la_3) (\la_1 - \la_2)\\
	&\; + 2 \s_1 (\s_1 - \la_2) (\la_1 - \la_3)
	 + \frac{\s_1^2 - \la_2^2}{\s_1 + \la_1} (- 2 \s_1 (\la_2 + \la_3) + 2 \la_2^2 + 2 \la_3^2)\\
	 &\; - \frac{\delta}{\la_1} (\s_1^2 - \la_2^2) (\s_1^2 - \la_2^2 + \s_1^2 - \la_3^2) \Big\}.
\end{aligned}
\]
For the first term inside the curly brackets, we compute that
\[
 (\s_1 - \la_2) (\s_1 + \la_1) =  \s_1^2 - \la_2^2 + (\s_1 - \la_2) (\la_1 - \la_2).
\]
We arrive at
\[
\begin{aligned}
	&\; I + II + IV\\
	\geq &\; \frac{1}{\s_1^2 - \la_1^2} \left\{ 
	 \frac{\s_1^2 - \la_2^2}{\s_1 + \la_1} \Big[ (\la_1 + \la_2) (\s_1 + \la_1) - 2 \s_1 (\la_2 + \la_3) + 2 \la_2^2 + 2 \la_3^2\Big] \right.\\
	 &\;\left. + (\la_1 + \la_2) (\s_1 - \la_2 )(\la_1 - \la_2) + 2 \s_1 (\s_1 -\la_3) (\la_1 - \la_2) \right.\\
	 &\; \left. + 2 \s_1 (\s_1 - \la_2) (\la_1 - \la_3)
	- \frac{\delta}{\la_1} (\s_1^2 - \la_2^2) (\s_1^2 - \la_2^2 + \s_1^2 - \la_3^2) \right\}.
\end{aligned}
\]

Define
\[
V := (\la_1 + \la_2) (\s_1 + \la_1) - 2 \s_1 (\la_2 + \la_3) + 2 \la_2^2 + 2 \la_3^2
\]
and
\[
\begin{aligned}
	VI : = &\; (\la_1 + \la_2) (\s_1 - \la_2 )(\la_1 - \la_2) + 2 \s_1 (\s_1 -\la_3) (\la_1 - \la_2) \\
	&\; + 2 \s_1 (\s_1 - \la_2) (\la_1 - \la_3) - \frac{\delta}{\la_1} (\s_1^2 - \la_2^2)  (\s_1^2 - \la_2^2 + \s_1^2 - \la_3^2).
\end{aligned}
\]
Therefore,
\[
I + II + IV \geq \frac{1}{\s_1^2 - \la_1^2} \left\{ 
\frac{\s_1^2 - \la_2^2}{\s_1 + \la_1} V + VI \right\}.
\]
We have
\[
\begin{aligned}
	V 
	= &\;  (\la_1 + \la_2) (2\la_1 + \la_2 + \la_3) - 2 \la_1 (\la_2 + \la_3) - 4 \la_2\la_3\\
	= &\; 2 \la_1 (\la_1 - \la_3) + (\la_1 + \la_2) (\la_2 + \la_3) - 4\la_2 \la_3\\
	= &\; 2 \la_1^2 - \la_1 \la_3 + \la_1 \la_2 + \la_2^2 - 3 \la_2 \la_3
\end{aligned}
\]
and
\[
\begin{aligned}
VI \geq &\; \frac{\s_1 - \la_2}{\la_1} \Big(\la_1 (\la_1^2 - \la_2^2) + \frac{2 \s_1 (\s_1 - \la_3) (\la_1 - \la_2) \la_1}{\s_1 - \la_2} \\
&\; + 2 \s_1 \la_1 (\la_1 - \la_3) - 2\delta \s_1^2 (\la_1 + \la_2)
- 2 \delta \s_1^2 (\la_2 + \la_3) \Big),
\end{aligned}
\]
where we used $\s_1 + \la_2 = \la_1 + \la_2 + \la_2 + \la_3$ 
and $\la_2^2 + \la_3^2 \geq 0$.
Let
\[
\tilde \la_i = \frac{\la_i}{\la_1},\; 1 \leq i \leq 3.
\]
Thus, 
it follows that
\[
\begin{aligned}
	\frac{VI}{\la_1^3} \geq &\; \frac{\s_1 - \la_2}{\la_1} \Big( 1 - \tilde \la_2^2 + \frac{2}{1 + \tilde \la_3} (1 + \tilde \la_2 + \tilde \la_3) (1 - \tilde \la_2^2) \\
	&\; + 2 (1 + \tilde \la_2 + \tilde \la_3) (1 - \tilde \la_3)
	- 2 \delta (1 + \tilde \la_2) (1 + \tilde \la_2 + \tilde \la_3)^2\\
	&\; - 2 \delta (\tilde \la_2 + \tilde \la_3) (1 + \tilde \la_2 + \tilde \la_3)^2 \Big).
\end{aligned}
\]
By Lemma \ref{key-2}, we see that
\[
2 \delta (\tilde \la_2 + \tilde \la_3) (1 + \tilde \la_2 + \tilde \la_3)^2
=  O\Big(\frac{\delta}{\la_1^3}\Big). 
\]
Let
\[\begin{aligned}
	VII : =&\; (1 - \tilde\la_2^2) \Big[1 + \frac{2}{1 + \tilde \la_3} (1 + \tilde \la_2 + \tilde \la_3)\Big] + O\Big(\frac{\delta}{\la_1^3}\Big) \\
	&\; + 2 (1 + \tilde \la_2 + \tilde \la_3) (1 - \tilde \la_3)
	- 2 \delta (1 + \tilde \la_2) (1 + \tilde \la_2 + \tilde \la_3)^2.
\end{aligned}\]
We then arrive at
\[
I + II + IV \geq \frac{1}{\s_1^2 - \la_1^2} \left\{ 
\frac{\s_1^2 - \la_2^2}{\s_1 + \la_1} V + \la_1^3 \frac{\s_1 - \la_2}{\la_1} VII \right\}.
\]

We now divide the analysis into three cases.

\textbf{Case 1:} $\la_2, \la_3 \geq 0$.
This means $0 < \tilde \la_2, \tilde \la_3 \leq  O(\frac{1}{\la_1^3})$.
We can see that
\[
V = (\la_1 + \la_2 )(\la_1 - \la_3) + \la_1^2 + \la_2^2 - 2 \la_2 \la_3 \geq (\la_2 - \la_3)^2 \geq 0.
\]
Moreover, it is straightforward to verify that
\[
VII = 3 + 2 - 2 \delta + O\Big(\frac{1}{\la_1^3}\Big) > 0
\]
for $0 < \epsilon < 1/4$ and sufficiently large $\la_1$.

\textbf{Case 2:} $\la_2 < 0$. Then $- \frac{1}{2} \leq \tilde \la_2 < 0$ and $0 < \tilde \la_3 < \frac{1}{2} + O(\frac{1}{\la_1^3})$.
It is easy to see that
\[
V \geq 2 \la_1^2 - \la_1 \la_3 + \la_1 \la_2 \geq \la_1 (\la_1 - \la_3) \geq 0.
\]
And for $VII$, we find that
\[
VII \geq \Big(1 - \frac{1}{4}\Big) \Big(1 + \frac{4}{3} \Big) + 1 - 2 \delta + O\Big(\frac{1}{\la_1^3}\Big) > 0
\]
for $0 < \epsilon < 1/4$ and sufficiently large $\la_1$.

\textbf{Case 3:} $\la_3 < 0$. Then $- \frac{1}{2} \leq \tilde \la_3 < 0$ and $0 < \tilde \la_2 < \frac{1}{2} + O(\frac{1}{\la_1^3})$.
Hence, we have 
\[
V = 2 \la_1^2 - \la_1 \la_3 + \la_1 \la_2 + \la_2^2 - 3 \la_2 \la_3 > 0.
\]
A direct computation yields that
\[
VII \geq 3 \Big(1 - \frac{1}{4}\Big) + 2 - 2 \delta \Big(1 + \frac{1}{2}\Big)+ O\Big(\frac{1}{\la_1^3}\Big) > 0
\]
for $0 < \epsilon < 1/4$ and sufficiently large $\la_1$.
This completes the proof of Step 2, and hence the proof of Lemma \ref{key}.
\end{proof}

\section{Proof of the main theorem}

Let $\Omega \subset \mathbb{R}^3$ be a domain.
Denote by
$M_2 (D^2 u)$ the quantity $m_2 (\lambda (D^2u))$.
Now consider the following Dirichlet problem
\begin{equation}
	\label{Dirichlet}
	\left\{\begin{aligned}
		M_2 (D^2 u) = &\; 1, &\; \mbox{in} \; \Omega,\\
		u = &\; 0, &\; \mbox{on} \; \partial \Omega.
	\end{aligned}
	\right.
\end{equation}

\begin{lemma}
	\label{Pogrelov}
	Suppose that $u : \Omega \rightarrow \mathbb{R}$ with $\lambda( D^2 u) \in P_2^{1/2}$ is a smooth solution to the Dirichlet problem \eqref{Dirichlet}. Then, we have
	\begin{equation}
		\label{pogrelov}
		(-u)^\beta \Delta u \leq C,
	\end{equation}
	where $C$ and $\beta$ depend only on $||u||_{C^0}$ and the diameter of the domain $\Omega$.
\end{lemma}

\begin{proof}
	Consider the following test function
	\[
	P = \log \lambda_{max} + \beta \log (-u) + \frac{1}{2} |x|^2.
	\]
	Suppose that $P$ attains its maximum at the point $x_0 \in \Omega$.
	We rotate the coordinates such that $u_{ij} = \lambda_i \delta_{ij}$ and 
	\[
	\lambda_1 = \max_{1\leq i \leq 3}\{\la_i\}.
	\] 
	Recall the formulas for differentiating the eigenvalues from \cite{Spruck}. For all $ 1 \leq i, j \leq 3$,
	\[
	\frac{\partial \la_1}{\partial u_{ij}} = \delta_{1i}\delta_{ij}
	\]
	and
	\[
	\frac{\partial^2 \la_1}{\partial u_{1j}^2 } = \frac{1}{\la_1 - \la_j},\;
	\frac{\partial^2 \la_1}{\partial u_{j1}^2 } = - \frac{1}{\la_j - \la_1},\; \mbox{if}\; \la_1 > \la_j.
	\]
	Differentiating $P$ at $x_0$, we obtain
	\begin{equation}
		\label{diff-1}
		0 = \frac{u_{11i}}{\la_1} + \frac{\beta u_i}{u} + x_i
	\end{equation}
	and
	\begin{equation}
		\label{diff-2}
		0 \geq \frac{u_{11ii}}{\la_1} + \sum_{\ell > 1} \frac{2 u_{1\ell i}^2}{(\la_1 - \la_\ell)\la_1} - \frac{u_{11i}^2}{\la_1^2} + \frac{\beta u_{ii}}{u} - \frac{\beta u_i^2}{u^2} + 1.
	\end{equation}
	Since
	\[
	M_2^{ii} u_{ii} = 3 M_2 (D^2 u) = 3
	\]
	and
	\[
	M_2^{ii} u_{11ii} = - M_2^{pq,rs} u_{pq 1} u_{rs 1},
	\]
	contracting \eqref{diff-2} with $M_2^{ii}$ and by the above two equations, we obtain 
	\begin{equation}
		\label{P1}
		\begin{aligned}
			0 \geq &\; - \frac{ M_2^{pq,rs} u_{pq 1} u_{rs 1}}{\lambda_1} + \sum_{\ell > 1} \frac{2 M_2^{ii} u_{1\ell i}^2}{(\lambda_1 - \lambda_\ell)\lambda_1}\\
			&\; - \frac{M_2^{ii} u_{11i}^2}{\lambda_1^2} + \frac{3 \beta }{u} - \frac{\beta M_2^{ii} u_i^2}{u^2} + \sum M_2^{ii},
		\end{aligned}
	\end{equation}
	
	By the critical equation \eqref{diff-1}, we have
	\begin{equation}
		\label{u_i^2}
	- \frac{\beta M_2^{ii} u_i^2}{u^2} \geq - \frac{2}{\beta} \frac{M_2^{ii} u_{11i}^2}{\lambda_1^2} - \frac{2}{\beta} M_2^{ii} x_i^2.
	\end{equation}
	On the other hand,
	note that
	\begin{equation}\label{M2nd}
	\sum_{p,q,r,s} M_2^{pq,rs} u_{pq1} u_{rs1} = \sum_{p,q} M_2^{pp, qq} u_{pp1} u_{qq1} + \sum_{p\neq q} M_2^{pq,qp} u_{pq1}^2
	\end{equation}
	and
	\[
	M_2^{pq,qp} = \frac{M_2^{pp} - M_2^{qq}}{\la_p - \la_q} \leq 0.
	\]
	We obtain that
	\begin{equation}\label{3rd-2}
	\begin{aligned}
		& - \sum_{p \neq q} M_2^{pq,qp} u_{pq 1}^2 + \sum_{\ell > 1} \frac{2 M_2^{11} u_{1\ell 1}^2}{\lambda_1 - \lambda_\ell}\\
		\geq &\; 2 \sum_{\ell > 1} \frac{M_2^{\ell \ell} - M_2^{11} }{\lambda_1 - \lambda_\ell} u_{11 \ell}^2 + \sum_{\ell > 1} \frac{2 M_2^{11} }{\lambda_1 - \lambda_\ell} u_{11 \ell}^2\\
		\geq &\; \sum_{\ell > 1} \frac{4}{3} \frac{M_2^{\ell \ell} u_{11 \ell}^2}{\lambda_1} \geq \Big(1 + \frac{2}{\beta}\Big) \sum_{\ell > 1}  \frac{M_2^{\ell \ell} u_{11 \ell}^2 }{\lambda_1} 
	\end{aligned}
	\end{equation}
	for $\beta \geq 6$.

	Substituting \eqref{u_i^2}, \eqref{M2nd} and \eqref{3rd-2} into \eqref{P1}, we obtain
	\begin{equation}
		\label{P2}
		\begin{aligned}
			0 \geq &\; - \frac{ M_2^{pp,qq} u_{pp 1} u_{qq 1}}{\lambda_1} + \sum_{i > 1} \frac{2 M_2^{ii} u_{i i 1}^2}{(\lambda_1 - \lambda_i)\lambda_1}\\
			&\; - \Big(1 + \frac{2}{\beta}\Big) \frac{M_2^{11} u_{111}^2}{\lambda_1^2} + \frac{3\beta}{u} + \Big( 1 - \frac{C}{\beta} \Big) \sum M_2^{ii},
		\end{aligned}
	\end{equation}
	where $C$ depends on the diameter of $\Omega$.
	By the key Lemma \ref{key}, for $\beta \geq 18$, we arrive at
	\begin{equation}
		\label{P3}
		0 \geq \frac{3\beta}{u} 
		+ \Big( 1 - \frac{C}{\beta} \Big) \sum M_2^{ii}.
	\end{equation}
Since
\[
\sum M_2^{ii} = \frac{2}{\la_1 + \la_2} + \frac{2}{\la_1 + \la_3} + \frac{2}{\la_2 + \la_3} \geq \frac{2}{\la_2 + \la_3} \geq \frac{\la_1^2}{2}
\]
by Lemma \ref{key-2}, we finally obtain
	\begin{equation}
		\label{P4}
	 (-u)\lambda_1^{2} \leq 12\beta
	\end{equation}
	for $\beta$ sufficiently large depending on the diameter of $\Omega$.
	This completes the proof.
\end{proof}

Now we are ready to prove Theorem \ref{thm}.
\begin{proof}	
	Define the level set
	\[
	\Omega_R := \{y\in \mathbb{R}^3 : u(Ry) \leq R^2\}.
	\]
	By the quadratic growth condition, for $y \in \Omega_{R}$ with $|y|\geq R_0/R$,  we derive
	$|y|^2 \leq \frac{C_2 + 1}{C_1}$.
	Therefore, the set $\Omega_R$ is bounded:
	\[
	\mbox{diam} (\Omega_R) \leq \max\{\frac{R_0}{R}, \frac{C_2 + 1}{C_1}\}.
	\]
	Define a new function $v$ on $\Omega_R$ by
	\[
	v (y) = \frac{u(Ry) - R^2}{R^2}.
	\]
	Then, $v$ satisfies 
	\begin{equation}
		\label{eqn-v}
		M_2 (D^2 v) = 1 \; \mbox{in}\; \Omega_R,\;
		\mbox{and}\; v = 0 \; \mbox{on}\; \partial \Omega_R.  
	\end{equation}
By the maximum principle, $a |x|^2 - b$ is a subsolution of the above equation for sufficiently large $a$ and $b$ depending on the diameter of $\Omega_R$.
	Hence , we get 
	\begin{equation}\label{C0}
		\sup_{\Omega_R}|v| \leq C,
	\end{equation}
	where $C$ depends on $\mbox{diam}(\Omega_R)$.
	By Lemma \ref{Pogrelov}, we obtain
	\[
	(-v)^{\beta} \Delta v \leq C
	\]
	for $C$ and $\beta$ depending on $C_1$ and $C_2$.
	
	Next, for $y$ belongs to
	\[
	\Omega_{\frac{R}{2}} := \{y\in \mathbb{R}^n : u(Ry) \leq \frac{R^2}{2}\},
	\] 
	we obtain that $v(y) \leq \frac{R^2/2 - R^2}{R^2} = - \frac{1}{2}$.
	Therefore, we derive that
	\[
	\Delta v (y) \leq 2^\beta C, \; \forall\; y\in \Omega_{\frac{R}{2}},
	\]
	where $C$ depends on $C_1$ and $C_2$.
	Note that $D^2_y v = D^2_x u$. The above inequality shows that
	\[
	\Delta u (x)\leq C, \; \forall \; x \in \{ x\in \mathbb{R}^n : u(x) \leq \frac{R^2}{2} \},
	\]
	for a uniform positive constant $C$. Letting $R$ go to infinity, we 
	have 
	\[
	\Delta u \leq C \; \mbox{on}\; \mathbb{R}^3, 
	\] 
	where $C$ is a uniform constant.
	Now, by the Evans-Krylov theorem \cite{GT} (see (17.41)),
	we have, for some $\alpha \in (0,1)$ and positive constant $C$ depending on $||D^2 u||_{C^0(\mathbb{R}^3)}$, that
	\[
	|D^2 u|_{C^\alpha (B_R)} \leq C \frac{\sup_{B_{2R}} |D^2 u|}{R^\alpha} \rightarrow 0 
    \]
	as $R \rightarrow + \infty$. Therefore, $u$ has to be a quadratic polynomial.

\end{proof}

\end{document}